\documentclass[11pt,a4paper]{amsart}

\makeatletter
\let\@wraptoccontribs\wraptoccontribs
\makeatother
\usepackage[english]{babel}
\usepackage{microtype}
\usepackage{amsmath}
\usepackage{amscd}
\usepackage{amssymb}
\usepackage{amsthm}
\usepackage[table]{xcolor}
\usepackage[matrix, arrow,curve]{xy}
\usepackage{verbatim}
\usepackage{lipsum}

\newtheorem{lemma}{Lemma}[section]
\newtheorem{prop}[lemma]{Proposition}
\newtheorem{thm}[lemma]{Theorem}
\newtheorem{cor}[lemma]{Corollary}

\theoremstyle{definition}
\newtheorem{defn}[lemma]{Definition}
\newtheorem{notation}[lemma]{Notation}

\newtheorem{ex}[lemma]{Example}

\newtheorem{rem}[lemma]{Remark}

\newcommand{\Z}{\mathbb{Z}}
\newcommand{\R}{\mathbb{R}}

\newcommand{\Sh}{\mathcal{S}}

\newcommand{\ohne}{\smallsetminus}
\newcommand{\ssm}{\smallsetminus}

\newcommand{\md}{\,\mathrm{d}} 
\newcommand{\vol}{\mathrm{vol}}
\newcommand{\reg}{\mathrm{Reg}}
\newcommand{\stokes}{\mathrm{Stokes}}
\newcommand{\fin}{\mathrm{fin}}
\newcommand{\glC}{\mathrm{s}C^1}
\newcommand{\sing}{\mathrm{sing}}

\newcommand{\df}{\mathrm{def}}

\begin{document}
\title{The period isomorphism in tame geometry}
\author{Annette Huber}
\address{Math. Institut, Albert-Ludwigs-Universität Freiburg, Ernst-Zermelo-Str.~1, 79102 Freiburg, Germany}
\email{annette.huber@math.uni-freiburg.de}
\contrib[with an Appendix joint with]{Johan Commelin}
\date{\today}
\maketitle

\begin{abstract}
We describe singular homology of a manifold $X$ via simplices
 $\sigma:\Delta_d\to X$ that satisfy Stokes' formula with respect to all differential forms. The notion is geared to the case of tame geometry (definable manifolds with respect to an o-minimal structure), where it gives a description of
the period pairing with de Rham cohomology via definable simplices.
\end{abstract}

In this note we close a gap in the literature on the period pairing. If $X$ is a differentiable manifold, there is a canonical isomorphism between de Rham cohomology and singular cohomology. It is induced from the \emph{period pairing} between de Rham cohomology and singular homology. The pairing has a good description
by integration
\[ (\sigma,\omega)\mapsto \int_{\Delta_d}\sigma^*\omega.\]
In order for this formula to make sense, the map $\sigma:\Delta_d\to X$ has to have good regularity properties. A good choice is to restrict to smooth maps.
If $X$ is a definable manifold for some o-minimal structure, e.g., if $X\subset\R^N$ is semi-algebraic, then the integral is absolutely convergent without any regularity assumptions. In \cite{period-buch}, this was used to give an alternative description of the set of period numbers in terms of semi-algebraic sets. Indeed, such a description is used as a definition for the notion of a period number in \cite{kontsevich_zagier}.

There are two problems that were not addressed in \cite{period-buch}: 
\begin{enumerate}
\item in order to get a well-defined pairing on homology, we need to establish Stokes' formula for semi-algebraic $\sigma$;
\item in order to show that the two pairings agree, we need to compare smooth and semi-algebraic $\sigma$'s.
\end{enumerate}
The same problems also appear in the setting of exponential periods treated in \cite{expper}, where it was side-stepped, see also Remark~\ref{comment} below.
We now present a conceptually clean solution. As in  \cite{expper}, we use input from the structure theory of definable sets: the existence of triangulations that are globally of class $C^1$ shown by Czap{\l}a--Paw{\l}ucki in \cite{omin-triang}. (An alternative is to apply the panel beating method of Ohmoto--Shiota \cite{ohmoto-shiota-triangulation} instead. It allows us to reparametrise a given semi-algebraic simplex as a $C^1$-map. There are some problems with this approach, see \cite[Section~7]{expper}.)

In the present note, we solve the two problems by introducing the notion of a simplex \emph{satisfying Stokes}. They are $C^1$ along open faces, all period integrals converge and satisfy Stokes' formula. We show that the complex built from these simplices computes singular homology. 

We also check change of variables and Stokes' formula for definable simplices without any regularity assumptions. This allows us to make the comparison we wanted.

To a great extent, the present paper is an exercise in exploiting the full strength of Czap{\l}a--Paw{\l}ucki's \cite{omin-triang}. Indeed, such applications to periods are the explicit rationale given by Ohmoto and Shiota in the introduction in the precursor \cite{ohmoto-shiota-triangulation}. Where we have to make analytic arguments in Section~\ref{sec:stokes}, they are on a much lower technical level.

\subsection*{Acknowledgement:} This note can be read as an addendum to the paper \cite{expper} with Johan Commelin and Philipp Habegger. I thank them for many discussions on issues related with period integrals and tame geometry. I am also indebted to my colleague Sebastian Goette for help with integrability computations. I am grateful to Tobias Kaiser and a referee for guiding  me through the o-minimal literature and both referees for their many comments and corrections.
The construction in the appendix is joint work with Johan Commelin. I thank him for allowing me to add it to the paper.

\section{Set-up}\label{sec:set-up}

Following \cite[4.6 (1)]{warner}, we define the \emph{standard $d$-simplex} as
\[ \Delta_d=\left\{(a_1,\dots,a_d)\in\R^d| a_i\geq 0, \sum a_i\leq 1\right\}.\]
By an \emph{open face} of $\Delta_d$, we mean the interior of a face (of any dimension) of $\Delta_d$. Throughout we are going to consider continuous
maps
\[ \sigma:\Delta_d\to\R^N\]
whose restriction to each open face is of class $C^1$. 


Most arguments center on the following construction, see \cite[p.~194]{warner}: 
to the simplex $\sigma$ we assign the cone
\begin{align*}
 \hat{\sigma}:\Delta_{d+1}&\to\R^N,\\
(a_0,a_1,\dots,a_d)&\mapsto \begin{cases} A\sigma(a_1/A,\dots,a_d/A),&
 A=\sum_{i=0}^da_i\neq 0,\\
 0 &A=0.
\end{cases}
\end{align*}
The simplex $\hat{\sigma}$ has a vertex at $0$ and the opposite face equal to $\sigma$. It interpolates linearly in between. The  map is again continuous (even for $A\to 0$) and $C^1$ on all open faces. This simplicial version of a homotopy can also be described via
\begin{align*}
\bar{\sigma}:[0,1]\times\Delta_d&\to\R^N\\
\bar{\sigma}:(t,b_1,\dots,b_d)&\mapsto (1-t)\sigma(b_1,\dots,b_d).
\end{align*}

Consider
\begin{align*}
	q \colon [0,1] \times \Delta_d &\to \Delta_{d+1} \\
	(t,b_1,\dots,b_d) &\mapsto ((1-t)(1-\sum_{i=1}^db_i), (1-t)b_1, \dots, (1-t)b_d).
\end{align*}
We have $\bar\sigma = \hat\sigma \circ q$ because with $a_0=(1-t)(1-\sum_{i=1}^db_i)$ and $a_i=(1-t)b_i$ for $i\geq 1$ we have $A=\sum_{i=0}^da_i=1-t$.

Note that $q|_{[0,1)\times\Delta_d}$ admits the inverse
$i:\Delta_{d+1}\ohne \{0\} \to [0,1)\times\Delta_d$ given by
\[  i:(a_0,a_1,\dots,a_d)\mapsto (1-A,a_1/A,\dots a_d/A).\]
 It extends to a diffeomorphism on an open neighbourhood of $\Delta_{d+1}\ohne \{0\}$. These formulas express (in the $d=1$ case) the fact 
that a triangle with one vertex removed can be reparametrised as square with one edge removed.

\section{Finite volume}

Let $\sigma:\Delta_d\to \R^N$ be as fixed in the last section, in particular continuous and $C^1$ along each open face.
\begin{defn}\label{defn:fvol} 
We say that
$\sigma$ has \emph{finite volume} if 
\begin{equation}\label{eq:volume} \int_{\Delta_d}\sigma^*(\omega)\end{equation} 
converges
absolutely for every continuous $d$-form $\omega$ on $\sigma(\Delta_d)$.
\end{defn}

The notion extends to chains (i.e.,  formal linear combinations of such maps): 
$\sum_{i=1}^na_i\tau_i$ is said to have finite volume if each $\tau_i$ has finite volume.

\begin{rem}The pull-back of $\omega$ to the interior of $\Delta_d$ is
a continuous $d$-form. The integral exists locally. Global existence, i.e,
convergence when approaching the boundary, is the issue. The condition is equivalent  to integrability (in the measure theoretic sense) of  $\sigma^*(\omega)$, in other words, whether the pull-back of $\omega$ is $L^1$. 
\end{rem}

\begin{lemma}\label{lem:suffice}It suffices to check that the integral (\ref{eq:volume}) converges absolutely for 
the standard $d$-forms $\md x_{i_1}\wedge\dots \wedge \md x_{i_d}$
for all $\{i_1,\dots,i_d\}\subset \{1,\dots,N\}$.
\end{lemma}
\begin{proof}We write 
\[ \omega=\sum_Ia_I\md x_I\]
where the sum is over multi-indices of length $d$. As $\omega$ is continuous, the functions
 $a_I$ are continuous on $\sigma(\Delta_d)$, in particular bounded. The pull-back is
\[ \sigma^*(\omega)=\sum_I a_I\circ\sigma\cdot \sigma^*(\md x_I).\]
By assumption, all $\sigma^*(\md x_I)$ are integrable. The $a_I\circ\sigma$ are bounded and continuous. This makes the sum integrable.
\end{proof}
\begin{rem}\label{rem:sm_suffice}
In particular, it does not matter if the convergence condition (\ref{eq:volume}) is imposed for $C^\infty$ differential forms or continuous forms.
If $d=N$, then our condition is indeed equivalent to finiteness of $\vol(\sigma(\Delta_N))$. As a referee pointed out: this makes it automatic in this case because $\sigma(\Delta_N)$ is compact.
\end{rem}

\begin{ex}
\begin{enumerate}
\item If $\sigma$ is $C^1$ globally on $\Delta_d$, then $\sigma^*(\md x_I)$ is $C^0$, in particular
integrable.
\item If $\sigma$ is semi-algebraic, or more generally definable in some o-minimal structure, then
\[ \int_{\Delta_d}\sigma^*(\md x_I)=\int_{\sigma(\Delta_d)}\md x_I\]
converges absolutely, see Proposition~\ref{prop:converge}. 
\end{enumerate}
\end{ex}

\begin{lemma}\label{lem:is_int}
A measurable differential form $\omega$ on $\Delta_{d+1}$ is integrable if and only if
$q^*(\omega)$ is integrable on $[0,1]\times \Delta_d$.
\end{lemma}
\begin{proof}The map $q$ is a diffeomorphism outside a set of measure $0$. 
The sign of the Jacobian determinant cannot change because it does not vanish for diffeomorphisms and $[0,1]\times \Delta_d$ is connected.
One of the integrals is finite if and only if the other is. 
\end{proof}

We need to check that the notion of having finite volume is stable under homotopies.

\begin{prop}\label{prop:fin_vol}
If $\sigma$ has finite volume, then so does the cone $\hat{\sigma}$.
\end{prop}
\begin{proof}
By Lemma~\ref{lem:is_int} it suffices to establish finiteness for $\bar{\sigma}$. By Lemma~\ref{lem:suffice} it suffices to consider the standard differential forms $ \md x_I$.

Let $\omega=\md x_1\wedge\dots\wedge \md x_{d+1}$ and $\omega_j$ the wedge product with the factor $\md x_j$ dropped. Let $\sigma=(\sigma_1,\dots,\sigma_N)$ and $\bar{\sigma}=(\bar{\sigma}_1,\dots,\bar{\sigma}_N)$, so that the $\sigma_i$ and $\bar{\sigma}_i$ are $\R$-valued functions. We compute  on $[0,1]\times\Delta_d$ in the coordinates of Section~\ref{sec:set-up}: 
\begin{align*}
\frac{\partial}{\partial t}\bar{\sigma}_j&=-\sigma_j\\
\frac{\partial}{\partial b_i}\bar{\sigma}_j&=(1-t)\frac{\partial}{\partial b_i}\sigma_j
\end{align*}
These are the entries of the Jacobian matrix. Hence
\[ \bar{\sigma}^*(\omega)=\sum_{j=1}^{d+1}(-1)^{j}(1-t)^d\sigma_j\md t\wedge\sigma^*(\omega_j).\]
It suffices to treat each summand separately. By assumption 
$\sigma^*(\omega_j)$ is integrable. The function $\sigma_j$ is continuous, making
$\sigma_j\sigma^*(\omega_j)$ integrable on $\Delta_d$. By Fubini this makes
$(1-t)^d\sigma_j\md t\wedge\sigma^*(\omega_j)$ integrable on $[0,1]\times \Delta_d$.
\end{proof}

\section{Stokes}\label{sec:stokes}

\begin{defn}\label{defn:stokes}
We say that $\sigma:\Delta_d\to \R^N$ \emph{satisfies Stokes} if
$\sigma$ and the chain $\partial \sigma$ have finite volume and for every $(d-1)$-form $\omega$ of class $C^1$ on a neighbourhood of $\sigma(\Delta_d)$ we have the formula
\[ \int_{\Delta_d}\sigma^*(\md \omega)=\int_{\partial \Delta_d}\sigma^*(\omega).\]
\end{defn}
The notion immediately extends to oriented simplicial complexes. The only case that we need is $[0,1]\times\Delta_d$.

Again we want to check that the condition is well-behaved under our homotopies.
\begin{lemma} Assume that $\sigma$ and $\partial\sigma$ have finite volume.
Then $\hat{\sigma}$ satisfies 
Stokes on $\Delta_{d+1}$ if and only if $\bar{\sigma}$ satisfies Stokes on
$[0,1]\times \Delta_{d}$.
\end{lemma}
\begin{proof}
Let $\omega$ be a smooth $(d+1)$-form on a neighbourhood of 
$\hat{\sigma}(\Delta_{d+1})$. We write $I=[0,1]$. Recall that $\bar{\sigma}=\hat{\sigma}\circ q$. As $q$ is a diffeomorphism on the interior of
$I\times\Delta_d$, we have
\[ \int_{I\times\Delta_d}\md \bar{\sigma}^*(\omega)=
 \int_{I\times \Delta_d} \bar{\sigma}^*(\md \omega)=
 \int_{\Delta_{d+1}}\hat{\sigma}^*(\md \omega)=\int_{\Delta_{d+1}} \md\hat{\sigma}^*(\omega).\]
It remains to show that
\[ \int_{\partial(I\times \Delta_d)}\bar{\sigma}^*(\omega)=\int_{\partial\Delta_{d+1}}\hat{\sigma}^*(\omega).\]
We compare the contributions of the boundary components. We have
\[ \partial(I\times\Delta_d)=(\partial I)\times\Delta_d+ I\times (\partial \Delta_d).\]
The boundary of $\Delta_{d+1}$ consists of the face $\tau^0$ opposite to
the origin and faces $\tau^1,\dots,\tau^d$ with vertex $0$. They are of the form
$\hat{F}$ for a face $F$ of $\tau^0$. We have
\begin{align*}
 q(\{0\}\times \Delta_d)&=\tau^0,\\
 q(I\times F)&=\hat{F}.
\end{align*}
As $q$ is a diffeomorphism (at least on the interior of these faces), this again implies
that the integrals of the pull-backs of $\omega$ agree.
Finally,
\[ q(\{1\}\times\Delta_d)=\{0\}.\]
As the map $\bar{\sigma}$ is constant on this face, the pull-back $\bar{\sigma}^*(\omega)$ vanishes and does not contribute to the boundary integral.
\end{proof}

We work on $I\times\Delta_d$ from now on (where $I=[0,1]$) and think of the first coordinate as time. The computation becomes clearer if we allow slightly more general objects than $\bar{\sigma}$.

\begin{lemma}\label{lem:decomp}
Let $\sigma=(\sigma_1,\dots,\sigma_N):\Delta_d\to\R^N$ be continuous and $C^1$ on all open faces, $f:I\to \R$  a $C^1$-function, and $\tau:I\times\Delta_d\to\R^n$ given
by 
\[ \tau(t,b_1,\dots,b_d)=f(t)\sigma(b_1,\dots,b_d).\]
Let 
\[ \eta=h\md x_1\wedge\dots\wedge \md x_d\]
with a continuous function $h$ on $\tau(I\times \Delta_d)$.
Then
\[ \tau^*(\eta)=A+B\]
where
\begin{align*}
A&= \tau^*(h)f^d\md\sigma_1\wedge\dots\wedge\md \sigma_d,\\ 
B&= 
\md t\wedge \frac{\partial f}{\partial t}f^{d-1}\tau^* (C),\\
C&=h\sum_i(-1)^{i-1}x_i \md x_1\wedge\dots \wedge \dot{\md x_i}\dots \wedge \md x_d.
\end{align*}
(Here $\dot{\md x}_i$ means that we omit this factor.)

The restriction to the faces of $I\times \Delta_d$ are
\[ \tau^*(\eta)|_{\{0,1\}\times \Delta_d}=A|_{\{0,1\}\times\Delta_d}\]
and
\[ \tau^*(\eta)|_{I\times F}=B|_{I\times F}\]
for all $(d-1)$-faces $F$ of $\Delta_d$.
\end{lemma}
\begin{proof} Let $\tau=(\tau_1,\dots,\tau_N)$. 
We have $\tau_i=f(t)\sigma_i$ and hence
\[ \md \tau_i=\frac{\partial f}{\partial t}\sigma_i\md t+f(t)\md \sigma_i.\]
This implies
\begin{multline*}
 \md\tau_1\wedge\dots \wedge \md\tau_d\\
=f(t)^d\md\sigma_1\wedge\dots \wedge \md\sigma_d+\sum_i(-1)^{i-1}\frac{\partial f}{\partial t}f(t)^{d-1}\sigma_i \md t\wedge \md\sigma_1\wedge \dots \wedge \md\dot{\sigma_i}\dots \wedge \md\sigma_d
\end{multline*}
where $\dot{\md\sigma}_i$ means that we omit this factor.
We introduce
\[ \omega=\md\sigma_1\wedge\dots\wedge \md\sigma_d,\quad \omega_i=\md\sigma_1\wedge \dots \wedge \md\dot{\sigma_i}\dots \wedge \md\sigma_d.\]
This allows us to write
\[ \md\tau_1\wedge\dots \wedge \md\tau_d=f(t)^d\omega +\sum_i(-1)^{i-1}\frac{\partial f}{\partial t}f(t)^{d-1}\sigma_i\md t\wedge \omega_i\]
and hence 
\begin{align*} \tau^*(\eta) &=(h\circ\tau)\md\tau_1\wedge\dots\md\tau_d\\
 &=(h\circ \tau )\left(f(t)^d\omega +\sum_i(-1)^{i-1}\frac{\partial f}{\partial t}f(t)^{d-1}\sigma_i\md t\wedge \omega_i\right)\\
&=(h\circ \tau )f(t)^d\omega +(h\circ\tau)\sum_i(-1)^{i-1}\frac{\partial f}{\partial t}f(t)^{d-1}\sigma_i\md t\wedge \omega_i
\end{align*}
We define the first summand as $A$ and the second as $B$. We then have
\[ B=\md t\wedge\frac{\partial f}{\partial t}f^{d-1}\tau^*(C)\]
as claimed.

We now restrict to faces. The restriction $\tau^*(\eta|_{\{0\}\times\Delta_d})$ is
\begin{multline*}
 \tau^*(h)|_{\{0\}\times\Delta_d}\md (f(0)\sigma_1)\wedge \dots\wedge \md(f(0)\sigma_d)\\
=\tau^*(h)|_{\{0\}\times\Delta_d}f(0)^d\omega=A|_{\{0\}\times\Delta_d}.
\end{multline*} 
The same computation applies to $t=1$.
Let $F$ be a $(d-1)$-face of $\Delta_d$. Then
\[ A|_{I\times F}=(\tau^*(h)f^d))|_{I\times F}\omega|_F=0\]
because $\omega$ is a $d$-form on a $(d-1)$-dimensional face. 
\end{proof}

\begin{prop}\label{prop:kegel_stokes}
If $\sigma$ satisfies Stokes, then so does $\hat{\sigma}$.
\end{prop}
\begin{proof}
It suffices to consider $\bar{\sigma}$. We write $\md=\md_t+\md_s$ for the decomposition into the time and space derivatives on $I\times\Delta_d$. Every smooth $d$-form on $\R^N$ decomposes as
\[ \sum_I h_Idx_I.\]
It suffices to consider each summand separately. Without loss of generality it
suffices to verify the formula for 
\[ \eta=h\md x_1\wedge\dots \wedge \md x_d\]
with smooth $h$. We apply the last lemma to $f(t)=(1-t)$ and $\tau=\bar{\sigma}$. Let $A$, $B$, $C$ be as defined there.

We claim that
\[ \int_{I\times \Delta_d}\md\tau^*(\eta)=\int_{\partial (I\times\Delta_d)}\tau^*(\eta).\]
The boundary chain satisfies
\[ \partial(I\times\Delta_d)=(\partial I)\times \Delta_d+ I\times (\partial \Delta_d).\]
The decomposition $\tau^*(\eta)=A+B$ in Lemma~\ref{lem:decomp} was constructed in a way that
$A$ vanishes on the boundary components in $I\times (\partial\Delta_d)$ and
$B$ vanishes on the boundary components in $(\partial I)\times \Delta_d$.  
It suffices to show that
\[ \int_{I\times\Delta_d}\md A=\int_{(\partial I)\times\Delta_d}A\]
and
\[ \int_{I\times\Delta_d}\md B=\int_{I\times(\partial \Delta_d)}B,\]
i.e., Stokes' formula holds for $A$ and $B$.

We have
\[ \md A=\md_t(\tau^*(h)\cdot f^d)\wedge\omega=\frac{\partial (\tau^*(h) \cdot f^d)}{\partial t}\md t\wedge \omega.\]
The partial derivative
is continuous on $I\times\Delta_d$ because $f$ and $h$ are  $C^1$ (actually smooth) and 
$\frac{\partial \tau_i}{\partial t}=\frac{\partial f}{\partial t}\sigma_i$ is continuous.  (Note that $\sigma_i$ does not depend on $t$! This is the decisive step in the proof.)  As $\sigma$ has finite volume, $\omega$ is integrable. Moreover, it is independent of $t$. This makes $\md t\wedge \omega$ and then $dA$ integrable
and we may apply Fubini.
We first integrate in the time direction, then in the spatial direction. 
By the fundamental theorem of calculus we may evaluate 
\begin{align*} \int_{I\times\Delta_d}\md A
&=\int_{\Delta_d}\int_I d_t(\tau^*(h)f^d)\wedge\omega\\
 & =\int_{\Delta_d}(\tau^*(h) f^d)|_{\{1\}\times\Delta_d}\omega-
      \int_{\Delta_d}(\tau^*(h)f^d)|_{\{0\}\times\Delta_d}\omega\\
&=\int_{(\partial I)\times \Delta_d} A.
\end{align*}
We also have
\[ B=\md t\wedge \frac{\partial f}{\partial t}f^{d-1}\tau^*(C), \quad \md B=\md t\wedge \frac{\partial f}{\partial t}f^{d-1}\md_s\tau^*(C).\]
Again we may apply Fubini because $\md B=\md\tau^*(\eta)-\md A$ is integrable. This time we take the integral in the spatial direction first. The differential form  $C$ on $\R^N$ is smooth. By assumption, $\sigma$ satisfies Stokes, hence so does $\tau=f(t)\sigma$ for fixed $t$. This gives
\[ \int_{\{t\}\times\Delta_d} \md_s\tau^*(C)
=\int_{\{t\}\times\partial \Delta_d}\tau^*(C)\]
and from this
\begin{multline*} 
\int_{I\times \Delta_d}\md B=\int_I \frac{\partial f}{\partial t}f^{d-1}
\left(  
\int_{\{t\}\times\Delta_d}\tau^*(\md_sC)\right)\md t  
\\
=\int_I\frac{\partial f}{\partial t}f^{d-1}\left(\int_{\{t\}\partial\Delta_d}\tau^*(C)\right)\md t=
\int_{I\times (\partial \Delta_d)}B.
\end{multline*}
\end{proof}

\section{Application to manifolds}\label{sec:manifolds}
While most readers will only be interested in the case of manifolds, the application to exponential periods in \cite{expper} involves manifolds with corners. We go through the definitions to clarify the terminology.

\begin{defn}\label{defn:diff}
Let $V\subset \R_{\geq 0}^N$ be an open subset, $\infty\geq p\geq 0$. A function
$f:V\to\R$ is called \emph{differentiable of class $C^p$} if it extends to a
$C^p$-function on an open neighbourhood of $V$ in $\R^{N}$.
\end{defn}
\begin{rem}
If $f$ is differentiable in the above sense, then all (higher) partial derivatives up to degree $p$ exist and are continuous on $V$, including on its boundary. They do not depend on the extension. Conversely, if the partial derivatives exist on $V$, then the function is differentiable by the \emph{Whitney extension problem} solved for compact sets by Feffermann in \cite{fefferman}. We do not need these subtleties.
\end{rem}

\begin{defn}\label{defn:manifold}
Let $X$ be a second countable Hausdorff space, $\infty\geq p\geq 1$, $N\geq 0$. A \emph{chart} $(U,\phi)$ of dimension $N$ on $X$ is a homeomorphism
\[ \phi:U\to V\]
with $U\subset X$ and $V\subset \R_{\geq 0}^N$ open subsets. Two charts $(U_1,\phi_1)$ and $(U_2,\phi_2)$ are \emph{compatible} if the transition map
\[ \phi_{12}=\phi_2^{-1}\circ \phi_1: \phi_1( U_1\cap U_2)\to \phi_2(U_1\cap U_2)\subset\R^{N}\]
is differentiable of class $C^p$ in the sense of Definition~\ref{defn:diff}. An \emph{atlas} is a set of compatible charts $(U_i,\phi_i)_{i\in I}$ such that $X=\bigcup_{i\in I}U_i$. The choice of an atlas defines the structure of a \emph{$C^p$-manifold with corners} on $X$.
\end{defn}

\begin{ex} The standard simplex $\Delta_d$ and $[0,1]\times\Delta_d$ are manifolds with corners.
\end{ex}

As usual, a continuous map between $C^p$-manifolds with corners is defined to be $C^p$ if the induced maps on charts are $C^p$ in the sense of Definition~\ref{defn:diff}. A special case are differential forms as sections of the cotangent bundle.

\begin{defn}Let $X$ be an $N$-dimensional $C^1$-manifold with corners, $\sigma:\Delta_d\to X$ continuous and $C^1$ on all open faces. We say that $\sigma$ has
\emph{finite volume} if there is a subdivision of $\sigma$ such that the
subsimplices map to a single chart each and have finite volume in the sense of Definition~\ref{defn:fvol}.
\end{defn}

\begin{lemma}The condition of having finite volume is independent of the choice of subdivision
and charts.
\end{lemma}
\begin{proof}Independence of the subdivision is obvious. Assume without loss of generality that $\sigma(\Delta_d)$ is contained in two charts
$\phi_i:U_i\to V_i\subset\R^N$ with $i=1,2$. Let $\phi_{12}=\phi_2\circ\phi_1^{-1}$ be the transition map where it is defined. By assumption it is $C^1$. 

Suppose that $\phi_1\circ\sigma$ has finite volume. We have to check that $\phi_2\circ\sigma$ has finite volume. By Remark~\ref{rem:sm_suffice} it suffices to consider 
\[ \sigma^*\phi_2^*(\omega)=\sigma^*\phi_1^*\phi_{12}^*(\omega)\]
for all smooth forms $\omega$ on $V_2$. The pull-back $\phi_{12}^*(\omega) $ is
a $C^0$-form on $V_1$. By assumption, its pull-back via $\phi_1\circ\sigma$ is integrable.
\end{proof}

If $\sigma$ has finite volume, the integral
\[ \int_{\Delta_d}\sigma^*(\omega)\]
is well-defined for all continuous $d$-forms on $X$. We can now replace $\R^N$ by
a $C^1$-manifold with corners $X$ in Definition~\ref{defn:stokes}. This defines what it means for $\sigma$ to \emph{satisfy Stokes}. The condition can be tested after passing to a subdivision of $\Delta_d$.

\begin{thm}\label{thm:fin_vol}
Let $X$ be a $C^1$-manifold with corners. Then the following complexes compute
singular homology of $X$:
\begin{itemize}
\item
the complex $S^\fin_\bullet(X)$ of formal linear combinations of singular simplices $\sigma$ such that all faces have finite volume;
\item the complex $S^\stokes_\bullet(X)$ of formal linear combinations of singular simplices $\sigma$ 
such that all faces satisfy Stokes.
\end{itemize}
\end{thm}
\begin{proof}Singular homology is defined as homology of the complex of all formal linear combinations of \emph{singular simplices}, i.e., continuous maps
$\sigma:\Delta_d\to X$ for $d\geq 0$.
The conditions on $S^\fin_\bullet(X)$ and $S^\stokes_\bullet(X)$ are stable under the boundary map, so we get well-defined subcomplexes. 

To see that the subcomplexes $S^\fin_\bullet(X)$ and $S^\stokes_\bullet(X)$ compute singular homology, we go through the argument in \cite[Section~5.31]{warner}. We sketch the argument for the convenience of the reader. We write $S^?_\bullet(X)$ for $?=\fin,\stokes$
 and $S_?^\bullet(X)$ for their $\Z$-duals. The assignment
\[ U\mapsto S_?^\bullet(U)\]
for $U\subset X$ open defines a complex of presheaves on $X$. Let $\Sh_?^\bullet$ be its sheafification. The argument of \cite[p.193/194]{warner} applies verbatim to show that each $\Sh^d_?$ is a fine sheaf. (The proof via \cite[5.22]{warner} uses paracompactness of $X$, not properties of the simplices.) 

In order to show that it is even a resolution of the constant sheaf $\Z$, it suffices to show that $S^?_d(U)$ is contractible in the special case where $U\subset\R_{\geq 0}^a\times\R^b$ is a ball centered at $0$. Warner achieves this by a simplicial homotopy in \cite[p.~194]{warner}. It is given by $\sigma\mapsto\hat{\sigma}$. Note that $\hat{\sigma}$ is contained in $U$ because $U$ is convex (this is the only point where the corners enter the argument). The faces of $\hat{\sigma}$ are $\sigma$ itself
and faces of the form $\hat{\tau}$ for a face $\tau$ of $\sigma$. In the case $?=\fin$, they all have finite volume by Proposition~\ref{prop:fin_vol}, making the homotopy well-defined. In the case $?=\stokes$, we use Proposition~\ref{prop:kegel_stokes} to the same end.

In both cases,
\[ \Z\to \Sh^\bullet_?\]
is now a fine resolution, hence
\[ H^i(X,\Z)=H^i(\Sh_?^\bullet(X)).\]
It remains to show that the canonical map
\[ S^\bullet_?(X)\to \Sh_?^\bullet(X)\]
comparing the sections of the presheaf and its sheafification
is a quasi-isomorphism. The argument of \cite[5.23]{warner} applies verbatim, once we note  that the barycentric subdivision of a simplex of finite volume or satisfying Stokes has the same property.
\end{proof}


\section{The tame case}\label{sec:tame}

We fix an o-minimal structure on the real field. By \emph{definable} we always mean definable with parameters in a fixed subfield $k$ of $\R$. 
This means that we have chosen a system of \emph{definable} subsets of $\R^n$  (with parameters in $k$) for all $n$, satisfying certain axioms, see \cite[Chapter~1,~Section~3]{D:oMin}.  An example is the theory of semi-algebraic sets of $\R^N$ for $N\geq 0$ defined over $k$. Our discussion was chosen to apply to this case.

In \cite[Chapter~3]{expper}, we study the notion of a definable $C^p$-manifold (with corners) for $\infty\geq p\geq 0$ and some basics of integration theory for differentiable forms. They are manifolds with corners in the sense of Definition~\ref{defn:manifold} with a finite atlas and definable transition maps.

\begin{rem} We include corners in the discussion because this is needed in the application to exponential periods, but this is not a critical point.
\end{rem}

\begin{rem}\label{rem:embedd_corners}
Kawakami proved that a definable $C^p$-manifold can be embedded into
$\R^N$ as a definable $C^p$-manifold, see \cite{kawakami}. It is likely that this extends to manifolds with boundary or corners. 
\end{rem}

%

A definable $C^p$-manifold with corners $X$ has a tangent and cotangent bundle in the category of definable $C^{p-1}$-manifolds with corners. For a definable subset $G\subset X$ a \emph{continuous $d$-form} $\omega$ is continuous section of $\Lambda^dT^*X$ over $G$, see \cite[Definition~3.11]{expper}.

Let $X$ be a definable manifold with corners and $G\subset X$ a definable subset of dimension $d$. We denote by $\reg_d(G)$ the locus where $G$ is a $C^p$-submanifold of $X$. For $p<\infty$, the subset is definable by  \cite[2.2]{vdDMiller:96}. The Cell Decomposition Theorem immediately implies that the complement has dimension strictly less than $d$, see also \cite[Lemma~3.8]{expper}.

A \emph{pseudo-orientation} on $G$ is the choice of an equivalence class of the data of a definable open subset $U\subset \reg_d(G)$ such that $\dim (G-U)<\dim G$ together with an orientation on $U$, see \cite[Definition~3.14]{expper}. Two such data are equivalent if they
agree on the intersection of the open sets.
 A \emph{pseudo-oriented definable set} is a definable set together with the choice of a pseudo-orientation. If $\omega$ is a continuous $d$-form
on a pseudo-oriented definable $G$, then we can define
\begin{equation}\label{eq:int} \int_G\omega:=\int_U\omega.\end{equation}
The value only depends on the equivalence class. 
\begin{rem} The notion can be extended to manifolds of class $C^\infty$.
A pseudo-orientation is defined as an equivalence class of an orientation on $U\subset \reg_d(G)$ as above for finite (and varying) degrees of differentiability.
Mutatis mutandis all results below still apply.
\end{rem}

\begin{prop}[{\cite[Theorem~3.23]{expper}}]\label{prop:converge}
If $G$ is compact, $\omega$ a continuous $d$-form, then the integral \eqref{eq:int} is absolutely convergent.
\end{prop}
\begin{proof} 
For the convenience of the reader, we sketch the argument of \cite{expper} as well as an alternative argument using the existence of strict $C^1$-trian\-gu\-la\-tions.

By a definable partition of unity subordinate to the atlas (or an inclusion/exclusion argument as in the proof of \cite[Theorem~3.23]{expper}) for a cover of $G$ by open balls whose closure is contained in a coordinate chart, we reduce to the affine case, i.e., $G\subset \R^N$.

The differential form $\omega$ is of the form $\sum_Ia_i\md x_I$ where the sum is with respect to all multi-indices of length $d$.  It suffices to consider each summand separately. As $G$ is compact and $a_I$ is continuous, the function is bounded. It suffices to establish absolute convergence for all $\md x_I$. Without loss of
generality $\omega=\md x_1\wedge \dots\wedge \md x_d$. From now on compactness of $G$ no longer matters and we replace $G$ by an oriented definable open subset representing the pseudo-orientation. In particular it is a $C^p$-manifold.

Let $\pi:\R^N\to \R^d$ be the projection to the first $d$ coordinates. Note
that $\omega=\pi^*(\md y_1\wedge\dots\wedge\md y_d)$.
By \cite[Lemma~3.9]{expper} the definable set $G$ decomposes (up to a set of smaller dimension) into definable open connected submanifolds $G_0,G_1,\dots,G_n$ such that $\pi$ induces a diffeomorphism $G_i\to \pi(G_i)\subset\R^d$ for $i=1,\dots,n$ and $\pi$ has positive fibre dimension over $\pi(G_0)$. This implies
$\omega|_{G_0}=0$, so that $G_0$ does not contribute to the integral. 
We also have $\int_{G_i}\omega=\vol(\pi(G_i))$. The latter is finite because
$G_i$ is relatively compact.

Alternatively, we use the same reduction to the affine case as above and then apply \cite[Main~Theorem]{omin-triang} to $G\subset \R^N$ and the closed
subset complementary to the set $U$ on which the orientation is defined. This yields a strict $C^1$-triangulation of $G$, in particular, all $\sigma:\Delta_d\to G$ are $C^1$ as maps to $\R^N$. 
Then the pull-back $\sigma^*(\omega)$ is $C^0$, and hence
\[ \int_{\sigma(\Delta_d)}\omega=\int_{\Delta_d}\sigma^*(\omega)\]
is finite.
\end{proof}

\begin{rem}
This is very similar to the arguments in \cite[(1.13)]{hanamura_et_al_I} in the semi-algebraic case.
\end{rem}

We now show that standard properties of integration extend to the pseudo-oriented case. Our first aim is the formula for change of variables.

\begin{defn}
Let $X_1,X_2$ be definable $C^p$-manifolds with corners for $1\leq p<\infty$, $G_1\subset X_1$ a definable subset of dimension $d$. Let $f:G_1\to X_2$ be a continuous definable map. We put
\[ \reg(f)=\left\{x\in \reg_d(G_1)| \text{$f$ is $C^p$ in a neighbourhood of $x$}\right\}.\]
\end{defn}
\begin{lemma}
The subset $\reg(f)\subset G_1$ is open, definable and $\dim(G_1-\reg(f))<d$.
\end{lemma}
\begin{proof} All properties can be checked in charts.
The condition is open. By \cite[B.7]{vdDMiller:96} the set is
definable. It remains to check the dimension property. By \cite[Lemma~3.8]{expper}, the set $G_1-\reg_d(G_1)$ has dimension smaller than $d$. We replace $G_1$ by
$\reg_d(G_1)$.
By \cite[C.2]{vdDMiller:96} there is a union $V$ of disjoint open $C^p$-cells in $G_1$ such that $\dim(G_1-V)<d$ and $f|_V$ is $C^p$.
\end{proof}

\begin{defn}Let $G_1\subset X_1$ and $G_2\subset X_2$ be pseudo-oriented definable subsets of dimension $d$ in definable $C^p$-manifolds with corners for $1\leq p<\infty$. 
 A continuous definable map
$f:G_1\to G_2$ is \emph{compatible with orientations} if there are representatives of the pseudo-orientations on $U_1\subset G_1$ and $U_2\subset G_2$ and a definable open subset $U\subset U_1\cap \reg(f)\cap f^{-1}(U_2)$ such that the map 
$U\to U_2$ is orientation preserving and $\dim(f(G_1-U))<d$.
\end{defn}
\begin{rem} Note that $G_2$ is not a manifold.
Here $\reg(f)$ refers to the regularity locus of the composition $G_1\to G_2\to X_2$. On the set $U_1\cap f^{-1}(U_2)\cap\reg(f)$, the induced map is $C^p$. By admitting the smaller set $U$, the notion becomes independent of the choice of representative for the pseudo-orientations.
\end{rem}

\begin{prop} Let $1\leq p<\infty$. Let $f:G_1\to G_2$ be a continuous definable map between
pseudo-oriented definable subsets of definable $C^p$-manifolds with corners. Then the change of variables formula holds, i.e., for any continuous differential form $\omega$ on
$G_2$, we have
\[ \int_{G_1}f^*(\omega)=\int_{f(G_1)}\omega\]
where $f(G_1)$ is pseudo-oriented as a subset of $G_2$. In particular, one side of the equation is absolutely convergent if and only if the other is.
\end{prop}
\begin{rem} \label{rem:small}
If $\dim(f(G_1))<d$, then the statement has to be understood as saying that the left hand side vanishes, see \cite[Remark~3.15]{expper}. 
\end{rem}
\begin{proof} We choose representatives for the pseudo-orientations.
Neither the left hand side nor the right hand side change if we replace $G_1$ by an open subset
such that the complement has dimension less than $d$. Without loss of generality,
$G_1=\reg(f)$ and $G_1$ is oriented. Let $U_2\subset G_2$ be an oriented definable open subset representing the pseudo-orientation of $G_2$. Let $U\subset G_1$ be as in the definition of compatibility of $f$ with orientation, i.e., open, contained in $f^{-1}(U_2)$ and $\dim(f(G_1- U))<d$. By the usual change of variables for the differentiable map $f|_U:U\to U_2$, we have
\[ \int_Uf^*(\omega)=\int_{f(U)}\omega.\]
One side is absolutely convergent if and only if the other is.

Let $G'=G_1-U$.
As $\dim( f(G'))<d$, we also have
\[ \int_{f(U)}\omega=\int_{f(G_1)}\omega.\]
It remains to show that
\[ \int_Uf^*(\omega)=\int_{G_1}f^*(\omega).\]
Let $G_0\subset G_1$ be the set of points in which
$f$ has positive fibre dimension. The Jacobian  of the differentiable map 
$f:G_1\to X_2$ has rank smaller than $d$ on $G_0$, hence $f^*(\omega)=0$ on $G_0$. The vanishing locus of $f^*(\omega)$ on $G_1$ is definable. We may remove it from $G_1$ without changing the value of the integral. In particular, we have removed the positive dimensional fibres, so from now on all fibres of $f:G_1\to G_2$ have dimension $0$. This implies that $\dim(f^{-1}( f(G'))<d$, hence also 
$\dim G'<d$ and
\[ \int_{G'}f^*(\omega)=0.\] 
\end{proof}

\begin{lemma}\label{lem:transport}
Let $f:G_1\to G_2$ be a definable homeomorphism between definable subsets of definable manifolds with corners. Given a pseudo-orientation on $G_1$ there is a unique pseudo-orientation on $G_2$ making $f$ compatible with orientations, and conversely.
\end{lemma}
\begin{proof}We may remove the complements of $\reg_d(G_1)$ and $\reg_d(G_2)$
as well as $\reg(f)$ and $\reg(f^{-1})$ from the situation. So without loss of generality
$G_1$ and $G_2$ are manifolds and $f$ is a diffeomorphism. We can then use
$f$ to transport the orientation.
\end{proof}

\begin{rem}The result is standard for oriented manifolds, even with boundary. Even though we usually think of orientations in terms of the tangent bundle, it is actually a completely topological notion that can be determined in terms of homology, see \cite[Section~3.3]{hatcher}.
\end{rem}

\begin{cor}\label{cor:int_simplex}
Let $1\leq p\leq \infty$. Let $X$ be a definable $C^p$-manifold with corners,
and $\sigma:\Delta_d\to X$ be  a definable continuous map, $\omega$ a $C^1$-form on a neighbourhood of $\sigma(\Delta_d)$.  Then $\sigma^*(\omega)$ is measurable and
\[ \int_{\Delta_d}\sigma^*(\omega)\]
is absolutely convergent.
\end{cor}
\begin{proof} Without loss of generality $p<\infty$. Apply Lemma~\ref{lem:transport} to the standard orientation
on the interior of $\Delta_d\subset\R^d$. This yields a pseudo-orientation on
$\sigma(\Delta_d)$ such that $\sigma$ is compatible with orientation. The form
$\sigma^*(\omega)$ is continuous on the regular locus of $\sigma$, hence measurable on $\Delta_d$. We then have
\[ \int_{\Delta_d}\sigma^*(\omega)=\int_{\sigma(\Delta_d)}\omega.\]
The right hand side is absolutely convergent by Proposition~\ref{prop:converge}.
\end{proof}
\begin{rem}
Note that $\sigma^*(\omega)$ is not continuous/bounded/defined everywhere unless we impose global differentiability conditions on $\sigma$!
\end{rem}


\section{A Theorem of Stokes in tame geometry}

After the preparation in the previous section, we are ready to address Stokes' formula.

\begin{prop}\label{prop:definable_stokes}
Let $2\leq p\leq \infty$. Let $X$ be a definable $C^p$-manifold with corners,
and $\sigma:\Delta_d\to X$ be  a definable continuous map, $\omega$ a $(d-1)$-form of class $C^p$ on a neighbourhood of $\sigma(\Delta_d)$.  Then Stokes' formula holds:
\[ \int_{\Delta_d}\sigma^*(\md \omega)=\int_{\partial \Delta_d}\sigma^*(\omega).\]
\end{prop}
\begin{proof} Both sides of the formula are well-defined by Corollary~\ref{cor:int_simplex}.
The values of the integral do not change when making $p$ smaller. Without loss of generality $p<\infty$, so that all results in the previous section apply.

We begin with the case $X=\R^N$. 
Let $\Gamma\subset \Delta_d\times\R^N$ be the graph of $\sigma$.
We choose a definable triangulation of $\Gamma$ that is globally $C^1$. It exists by Czap{\l}a-Paw{\l}ucki \cite[Main Theorem]{omin-triang}.  The projection to the first factor is a triangulation $(K,\Phi)$ of $\Delta^d$ such that both $\Phi$ and $\sigma\circ\Phi$ are globally $C^1$.
The orientation on $\Delta_d$ defines an orientation on $|K|$ and all $d$-simplices in $K$.   By change of variables for $\Phi$

\[ 
\int_{\Delta_d}\sigma^*(\md \omega)=\int_{|K|}\Phi^*\sigma^*(\md \omega),\quad 
\int_{\partial \Delta_d}
\sigma^*(\omega)=\int_{\partial|K|} \Phi^*\sigma^* (\omega).
\]
%
%

Let
\[ \tau:\Delta_d\to |K|\]
be a simplex of $K$. The map $\sigma\circ\Phi\circ\tau$ is $C^1$ (globally, not only on open faces).
 By Stokes' theorem (in the $C^1$-version of Whitney in \cite[Chapter~III, \S\S 16-17]{whitney}, see \cite[Theorem~1.4]{expper} for details) we have
\[ \int_{\Delta_d}\tau^*\Phi^*\sigma^*(\md \omega)=\int_{\partial \Delta_d}\tau^*\Phi^*\sigma^*(\omega).\]
We sum over all $d$-simplices in $K$. 
 The union of the interiors of all $d$-simplices
$\bigcup_\tau \tau(\Delta_d^\circ)$ is open and dense in $|K|$, hence the sum is
\[ \sum_{\tau\in K_d}\int_{\Delta_d}\tau^*\Phi^*\sigma^*(\md \omega)=\int_{|K|}\Phi^*\sigma^*(\md \omega).\]
Now consider the sum over the boundaries. Two things can happen:
if $F$ is an open $(d-1)$-simplex of $K$, then it is either fully contained
in the interior of $|K|$ or fully contained in $\partial |K|$. In the first case, there is a second $d$-simplex with face $F$, but opposite orientation. These contributions cancel. In the second case, $F$ is part of a triangulation of $\partial |K|$.  Hence the sum gives
\[\sum_{ \tau\in K_d}\int_{\partial \Delta_{d}}\tau^*\Phi^*\sigma^*(\omega)=\int_{\partial |K|}\Phi^*\sigma^*(\omega).\]
Putting the equalities together, we have Stokes' formula for $\sigma$ for $X=\R^N$.

For general definable manifolds  with corners $X$ with finite atlas $(U_i,\phi_i)_{i\in I}$, we may apply repeated barycentric subdivision of $\Delta_d$ into linear subsimplices
$\delta_1,\dots,\delta_M$
 such that the image of each subsimplex $\delta_j$ is contained in a single chart $U_{i(j)}$. We apply Stokes' formula to
$\phi_{i(j)}\circ\sigma|_{\delta_j}$. By change of variables with respect to $\phi_{i(j)}$, this gives Stokes' formula for $\sigma|_{\delta_j}$. Again by adding up the contributions of all $\delta_j$ we get the formula for $\sigma$. 
\end{proof}

\begin{rem}\label{rem:stokes-literatur}
There is already a rich literature on Stokes' Theorem in tame settings. 
To name a few: Paw{\l}ucki, see \cite{pawlucki-stokes}, works in the sub-analytic setting;  Funk, see \cite{funk}, in the semi-algebraic setting and most recently Julia, see \cite{julia-thesis} and \cite{julia-stokes}, includes the o-minimal case. The latter two authors use the more modern language of integral currents. 

These results have a different flavour from the above. The differential form is assumed to be continuous (still bounded in the more general \cite[Theorem~A]{julia-stokes}), excluding $\sigma^*(\omega)$ as above.  

Instead their results imply (as a very special case) a Stokes' formula 
\[ \int_{\sigma(\Delta_d)}\md \omega=\int_{\Sigma}\omega\]
where $\Sigma$ is a suitable definable set of dimension $d-1$. (The quasi-regular points of the frontier of the regular locus in the case of \cite{pawlucki-stokes}; constructed from a cell decomposition in \cite{julia-thesis}). 

In order to deduce Proposition~\ref{prop:definable_stokes}, one would need to relate $\sigma(\partial\Delta_d)$ and
$\Sigma$. This is probably doable using the strict $C^1$-triangulations that the above proof uses to give the proof directly. 
\end{rem}

We can recover a similar version of Stokes' Theorem, at least for $C^1$-forms.

\begin{prop}\label{prop:definable_stokes_II}
 Let $X$ be a definable $C^p$-manifold with corners for $2\leq p<\infty$.
Let $G\subset X$ be a pseudo-oriented compact definable subset of dimension $d$ and
 $\omega$ a $(d-1)$-form of class $C^1$ on a neighbourhood of $G$. 
Then there is a formal $\Z$-linear combination of pseudo-oriented compact definable subsets of $G$ of dimension $(d-1)$ such that
\[ \int_{G}\md \omega=\int_\Sigma \omega.\]
\end{prop}
\begin{proof}Let $U\subset G$ be  a definable oriented subset representing the 
pseudo-orientation. By \cite[Proposition~7.6]{expper} (or directly \cite{omin-triang} for $X=\R^N$), there is a definable triangulation $(K,\Phi)$ of $X$ compatible with $\bar{U}$ and $\bar{U}-U$ which is globally of class $C^1$. Let
$\tau_1,\dots,\tau_n$ be the $d$-dimensional simplices contained in $\bar{U}$. Their interior is contained in $U$.
Hence all $\tau_i$ inherit an orientation from $U$. We define $\Sigma$ as the chain
\[ \Sigma=\sum_i \partial\tau_i.\]
As in the proof of Proposition~\ref{prop:definable_stokes} we apply Whitney's
$C^1$-version of Stokes' Theorem and obtain
\[ \int_G\md \omega=\sum_i \int_{\Delta_d}\md \tau_i^*(\omega)=\sum_i \int_{\partial \Delta_i}\tau_i^*(\omega) =\int_\Sigma\omega.\]
\end{proof}

\begin{rem} In the language used in \cite{julia-thesis}: the set
$G$ defines a current $T$ and $\partial T$ is definable because it is represented by $\Sigma$. This is Julia's \cite[Theorem~6.3.2]{julia-thesis}. He has to make a subtle limit argument. This difficulty has not disappeared---it is contained in \cite{omin-triang}.
\end{rem}

\section{Period isomorphisms}
We are now ready to show that the period pairing for definable manifolds can be computed via definable simplices. As in Section~\ref{sec:tame} we fix an o-minimal structure on the real field. By \emph{definable} we always
mean definable with parameters in a fixed subfield $k$ of $\R$. Throughout this section, let
$X$ be a definable $C^\infty$-manifold with corners in the sense of Sections~\ref{sec:manifolds} and \ref{sec:tame}.

\begin{notation}We set:
\begin{itemize}
\item $S_d^\sing(X)$ the free abelian group with  basis continuous maps
$\sigma:\Delta_d\to X$;
\item $S_d^\infty(X)$ the free abelian group with basis smooth maps
$\sigma:\Delta_d\to X$;
\item $S_d^\stokes(X)$ the free abelian group with basis continuous maps $\sigma:\Delta_d\to X$ which are $C^1$ on all open faces and such that all faces have finite volume (see Definition~\ref{defn:fvol}) and satisfy Stokes (see Definition~\ref{defn:stokes});
\item $S_d^{\df,C^1}(X)$ the free abelian group with basis definable continuous maps $\sigma:\Delta_d\to X$ which are $C^1$ on all open faces;
\item $S_d^\df(X)$ the free abelian group with basis definable continuous maps $\sigma:\Delta_d\to X$;
\item $A^d(X)$ the space of all smooth $d$-forms on $X$.
\end{itemize}
\end{notation} 

In each case, the groups organise into a complex with the standard differential for singular homology and de Rham cohomology, respectively. The complexes $S^\df_\bullet(X)$ are functorial for all continuous definable maps between definable manifolds with corners. 

\begin{rem}\label{rem:defn_hom}
By definition, the homology of $S_\bullet^\sing(X)$ is the \emph{singular homology}.
The complexes $S^\df_\bullet(X)$ are the ones appearing
already in the definition of \emph{definable homology} in \cite{edmundo-woerheide} by Edmundo--Woerheide (in fact the very special case of an o-minimal structure on the real field with field of constants a subfield of $\R$).  I thank a referee for making me aware of their work.
\end{rem}

In applications, it is often helpful to pass to a subcomplex with a finite basis, even if functoriality is lost. This is where simplicial homology comes in. Definable triangulations exist for compact definable spaces, see the discussion in the appendix.

\begin{notation} Assume that $X$ is compact and let $(K,\Phi)$ be a definable triangulation of $X$. We set
\begin{itemize}
\item $S^\Delta_d(X)$ the free abelian group with basis the elements of
$K_d$.
\end{itemize}
\end{notation}
We obtain a complex $S_\bullet^\Delta(X)$ with the differential of simplicial homology. It is well-known to compute singular homology of $X$, see for example \cite[Theorem~2.3.10]{period-buch}.

\begin{prop}The inclusions
\[\begin{xy}
\xymatrix{
&&S_\bullet^\sing(X)\\
&S_\bullet^\df(X)\ar@{^{(}->}[ru]&&S_\bullet^\stokes(X)\ar@{_{(}->}[ul]\\
S_\bullet^\Delta(X)\ar@{^{(}->}[ru] &&S_\bullet^{\df,C^1}\ar@{_{(}->}[lu]\ar@{^{(}->}[ru] &&S_\bullet^\infty(X)\ar@{_{(}->}[lu]
}
\end{xy}\]
are natural quasi-isomorphisms, with the caveat that $X$ has to be compact for the comparison with simplicial homology and that the inclusion
of $S_\bullet^\Delta(X)$ only has the functoriality of simplicial homology.
\end{prop}
\begin{proof}All complexes compute singular homology. We mentioned a reference for simplicial homology above. For the Stokes version 
this is Theorem~\ref{thm:fin_vol}. The same proof also applies in the definable cases $S^\df_\bullet(X)$ and $S^{\df,C^1}(X)$ because $\hat{\sigma}$ is definable (and $C^1$ on open faces) if $\sigma$ is. Alternatively, we use Remark~\ref{rem:defn_hom} and apply the main comparison theorem of Edmundo and Woerheide in \cite{edmundo-woerheide} to singular homology and definable homology.

The case of $S^\infty_*(X)$ is in \cite[Section~5.31]{warner}. 
\end{proof}

The \emph{period pairing}
\[ S^\infty_d(X)\times A^d(X)\to \R,\quad(\sigma,\omega)\to\int_{\Delta_d}\sigma^*\omega\]
induces a pairing of complexes by Stokes' theorem. Stokes' theorem also holds for $S^{\stokes}_*(X)$, its subcomplex $S^{\df,C^1}(X)$  and for $S^\df_*(X)$. We also get well-defined pairings of complexes in these cases.

\begin{thm}\label{main}Let $X$ be a definable manifold with corners and $(K,\Phi)$ a definable triangulation of $X$. Then the period pairing can be computed by integration on continuous definable simplices. In other words, the pairing extends to a pairing of complexes of $A^\bullet(X)$ with $S_\bullet^\stokes(X)$, $S^{\df,C^1}_\bullet(X)$, $S_\bullet^\df(X)$  and $S_\bullet^\Delta(X)$ in a 
compatible way with the quasi-isomorphisms
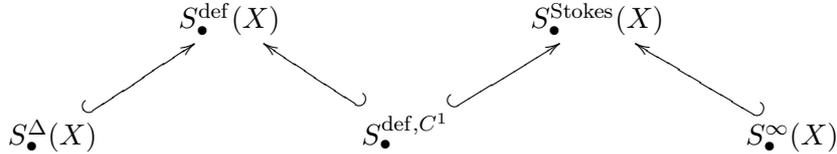
\[\begin{xy}
\xymatrix{
&S_\bullet^\df(X)&&S_\bullet^\stokes(X)\\
S_\bullet^\Delta(X)\ar@{^{(}->}[ru]& &S_\bullet^{\df,C^1}\ar@{_{(}->}[lu]\ar@{^{(}->}[ru] &&S_\bullet^\infty(X)\ar@{_{(}->}[lu]
}
\end{xy}\]
\end{thm}
\begin{proof}
In each case, the pairing is given by integration, hence the pairings are compatible. 

To make the pairings well-defined as pairings of complexes,
we have to check Stokes' formula. For $S^{\stokes}_\bullet(X)$ it holds by assumption. For its subcomplex $S^{\df,C^1}_\bullet(X)$ and the full $S^\df_\bullet(X)$ it holds by Proposition~\ref{prop:definable_stokes}. It holds for $S^\infty_\bullet(X)$ as a subcomplex of $S^\stokes_\bullet(X)$ (which is a complicated way of saying that the standard Theorem of Stokes holds for smooth simplices). It holds
for $S_\bullet^\Delta(X)$ as a subcomplex of $S^\df_\bullet(X)$.
\end{proof}

\begin{rem}\label{comment}
In the context of periods in the number theoretic sense 
(see \cite{period-buch}) 
or exponential periods (see \cite{expper}), both appearing in \cite{kontsevich_zagier}, one wishes to represent all homology classes of compact definable manifolds with corners (for certain o-minimal structures) by definable simplices and express the abstract period pairing by integration. This is achieved in \cite{expper}, but the above Theorem~\ref{main} is conceptually clearer and more flexible.

In \cite{expper}, we describe singular homology by  simplices $\sigma:\Delta_d\to X$ which are globally $C^1$ (or strictly $C^1$ in the terminology of \cite{omin-triang}), not only on open faces. They satisfy Stokes (in the sense of the present paper) by the $C^1$-version of Stokes' Theorem proved by Whitney in \cite{whitney} (see also the proof of Proposition~\ref{prop:definable_stokes} for details). Let us denote their complex by $S^{\glC}_\bullet(X)$. This implies that the period pairing $S^{\glC}_\bullet(X)\times A^\bullet(X)\to\R$ is given by integration. 

In a second step  \cite[Proposition~7.6]{expper} exploits the full strength of
\cite{omin-triang} to construct a definable strict $C^1$-triangulation of $X$: the Main Theorem of \cite{omin-triang}  is applied to a family of graphs  of a partition of unity subordinate to the charts of the manifold. 
Alternatively, we could invoke \cite{kawakami} to embed $X$ into $\R^N$ as a $C^1$-manifold and
apply the existence of $C^1$-triangulations of \cite{omin-triang} (in the strict sense) to $\R^N$.

In the notation of the present paper: There is
a definable triangulation $(K,\Phi)$ of $X$ such that all elements of $\Phi$ are globally $C^1$.
The inclusion $S^\Delta_\bullet(X)\to S_\bullet^\glC(X)$ is a quasi-isomorphism.
This achieves the aim.

Why should we extend to all definable simplices? The $C^1$-condition is artificial given that the integral converges in general. Theorem~\ref{main} is conceptually clearer.

On a more technical level: the functoriality of the constructions in \cite{expper} is bad because simplicial homology only becomes functorial when passing to homology - not on the level of complexes. This causes technical problems when generalising the results from manifolds to singular spaces and to relative homology---something that is very much needed in \cite{expper} as well.
Theorem~\ref{main} is more flexible.
\end{rem}

\begin{appendix}

\section{Existence of triangulations}
\begin{center}by Johan Commelin and Annette Huber\end{center}

As in Section~\ref{sec:tame} we work in the setting of definable manifolds
in a fixed o-minimal structure on the real field.
By \emph{definable} we always mean definable with parameters in a fixed subfield $k$ of $\R$. 

The existence of triangulations is known for definable sets. We extend this to
the manifold setting.

\begin{prop}\label{prop:triangle_manifold}
 Let $0\leq p<\infty$ and $X$ be a compact definable $C^p$-manifold with corners,
 $A_1,\dots,A_M$ definable subsets of $X$.
 Then there is a definable triangulation of $X$ compatible with $A_1,\dots,A_M$
which is $C^p$ on open faces and
 such that every simplex is contained in a chart.
\end{prop}
\begin{rem} 
We are going to give two proofs, both reducing the result to the affine case, i.e., submanifolds of $\R^N$.  This is possible for semi-algebraic Hausdorff spaces (the $C^0$-case) by \cite{Robs83}. In \cite[Proposition~7.6]{expper} this is combined with the strong
\cite{omin-triang} to deduce the existence of a triangulation in which all simplices are globally $C^1$.  The same reasoning (using \cite{shiota-book} instead of \cite{omin-triang}) also gives a triangulation that is $C^p$ on all open faces.

Robson's result was generalised to arbitrary o-minimal structures on $\R^N$ by Kawakami, see \cite{kawakami}, making the above subtleties unnecessary. This leads to the first proof below. We thank Tobias Kaiser and a referee for pointing the reference out to us.
\end{rem}

The following is well-known:
\begin{lemma}\label{lem:triangle_affine}
Proposition~\ref{prop:triangle_manifold} holds true for definable bounded subsets $X\subset\R^N$ and definable subsets $A_1,\dots,A_M$.
\end{lemma}
\begin{proof}The semi-algebraic case is \cite[Remark~9.2.3 a)]{BCR}. 
Shiota's  \cite[Chapter~II, Theorem~II]{shiota-book} contains the o-minimal case: An o-minimal structure on the real field with parameters in a subfield is an example of a system of $\mathfrak{X}$-sets and his notion of an $\mathfrak{X}$-triangulation in \cite[p.~96]{shiota-book} encompasses the differentiability assumption that we claim.

The existence of the triangulation is also a consequence of the much deeper recent \cite{pawlucki-cp}.
\end{proof}

\begin{proof}[First proof (without boundary).]
By \cite{kawakami} every $C^p$-manifold can be embedded into $\R^N$. 
Ka\-wa\-ka\-mi does not address manifolds with boundary or corners, but see Remark~\ref{rem:embedd_corners}. The image is bounded because $X$ is compact.
We then apply the affine triangulation result of Lemma~\ref{lem:triangle_affine} to $\R^N$ and the definable subsets $X$, $A_1,\dots,A_n$.
\end{proof}

We will give a second proof that uses less o-minimal analysis, but a computation with simplicial complexes.

\begin{lemma}\label{lem:triangle_extend}
 Let $0\leq p<\infty$.
 Let $X_1$ and $X_2$ be compact definable subsets of
 some ambient definable $C^p$-manifold with corners.
 Denote by $X$ the union $X_1 \cup X_2$ and by $B$ the intersection $X_1 \cap X_2$. Let $A_1,\dots,A_M\subset X$ be definable subsets.
 Assume that $X_1$ has a definable triangulation compatible with $B$ and all $A_i\cap X_1$ which is $C^p$ on all open faces of simplices. Assume that $X_2$ is affine.
 Then there is a definable triangulation of $X$ compatible with $A_1,\dots,A_M$  which is $C^p$ on all open faces of simplices.
\end{lemma}
\begin{proof}
The explicit formulas in the proof are based on the standard coordinates
on $\Delta_d$ as the interior of the convex hull of the basis vectors $e_0,\dots,e_d$ in $\R^{d+1}$. This deviates from Section~\ref{sec:set-up}.

 We start with a
 definable triangulation $\mathcal{T}_1=(K_1,h_1)$ of $X_1$ compatible with~$B$  and all $A_i\cap X_1$ which is $C^p$ on open faces.

 Note 
 that for every set of vertices $v_0,\dots,v_n \in
 K_1$ there is at most one open $n$-simplex with these vertices because this is the case for a simplicial complex.  We will write $(v_0,\dots,v_n)$ for this simplex and $(h_1(v_0),\dots,h_1(v_n))$ for its image in $X_1$.
 Without loss of generality we may assume that
 for every simplex $(v_0,\dots,v_n)$ in $K_1$ such that
 $h_1(v_0),\dots,h_1(v_n)$ lie in~$B$
 the entire simplex $(v_0,\dots,v_n)$ lies in~$B$
 (pass to a barycentric subdivision if necessary). 
 Now choose a triangulation $\mathcal{T}_2=(K_2,h_2)$ of $X_2$
 compatible with the images of elements of $\mathcal{T}_1|_B$ and all $A_i\cap X_2$ that is $C^p$ on open simplices.
 It exists by Lemma~\ref{lem:triangle_affine} because $X_2$ is affine.
 Again we may assume that
  if the images of the vertices of a simplex are in $B$, 
 then so is the image of the simplex. On $B$, the triangulation 
$\mathcal{T}_2$ ``subdivides'' $\mathcal{T}_1$.
 It remains to modify $\mathcal{T}_1$ on $X_1\ssm B$
 in such a way that the triangulations become compatible.

 We will now construct a set $K\subset K_1\times K_2$ of simplices, as follows.
 For every simplex
\[ \sigma=(v_0,\dots,v_m,b_0,\dots,b_n)\in K_1\]
 with $h_1(b_0),\dots,h_1(b_n)$ in~$B$ and $h_1(v_0),\dots,h_1(v_m) \notin B$,
 and for every simplex $\tau = (w_0,\dots,w_s) \in
 \mathcal{T}_2|_{(h_1(b_0),\dots, h_1(b_n))}$
 we add $(v_0,\dots,v_m,w_0,\dots,w_s)$ to~$K$.

 We make some remarks about this construction:
 \begin{itemize}
  \item The condition $\tau \in \mathcal{T}_2|_{(h_1(b_0),\dots, h_1(b_n))}$ is meaningful,
   because the image of the entire simplex $(h_1(b_0),\dots,h_1(b_n))$
   is contained in~$B$, by assumption.
   In particular, we get a triangulation  of~$h_1(\bar \sigma)$
   that is also compatible with the same construction on the faces of~$\sigma$.
  \item By taking $m = 0$,
   we see that $K$ contains all simplices of~$K_2$.
  \item Similarly, by taking $n = 0$,
   we see that $K$ contains all simplices of~$K_1$
   that do not have faces in~$B$.
 \end{itemize}

The next step is to  define the triangulation map $h:|K|\to X$. We do this on closed simplices in such a way that the definition is compatible with restriction to the faces. We fix $\sigma$ and $\tau$. 

 Since $X_1$ is compact, the closed simplex $\bar{\sigma}\in K_1$ can be identified with the standard simplex. We are given a definable map
$h_1|_{\bar{\sigma}}\colon\bar{\Delta}_{m+n+1}\to X_1$ that is a homeomorphism onto its image.
 The simplex $\tau$ gives a definable map $h_2|_{\bar{\tau}} \colon \bar{\Delta}_{s} \to B$,
 again a homeomorphism onto its image. Consider
 $g = h_1|_{\bar{\sigma}}^{-1} \circ h_2|_{\bar{\tau}} \colon \bar\Delta_{s} \to \bar\Delta_{m+n+1}$.
 We define a new map
 $h:\bar\Delta_{s+m+1} \to X_1$
 by mapping 
 \[
 \sum_{i=0}^{s+m+1}a_ie_i\mapsto h_1|_{\bar{\sigma}}\left(\sum_{i=0}^m a_ie_i
	+ a g\left(\frac{1}{a}\sum_{i=0}^{s} a_{i+m+1}e_i\right)\right)
 \]
	where the scaling factor $a$ is defined to be $\sum_{i=0}^{s-1} a_{i+m+1}$. For this, we check the limit when $a$ tends to $0$.
The value
of $g$ is bounded, hence $ag(\cdot)$ tends to $0$, when $a\to 0$. We apply the continuous function $h_1$, so the limit is
\[   h_1|_{\bar{\sigma}}\left(\sum_{i=0}^m a_ie_i \right).\]
The map $h$ takes the vertex $v_i$ (identified with $e_i$ in the formula) to $h_1(v_i)$
and the vertex $w_j$ (identified with $e_{j+m+1}$ in the formula) to $h_2(w_j)$.
The map is clearly definable and $C^p$ on all open faces. 
\end{proof}

\begin{proof}[Second proof of Proposition~\ref{prop:triangle_manifold}.]
 Let $U_1,\dots,U_n \subset X$ be an open cover of an atlas,
 $\phi_i \colon U_i \to V_i$ the charts
 with $V_i \subset \R_{\geq 0}^{n_i}\times\R^{m_i}$ open definable.
 In particular, the transition maps are $C^p$ and definable.

 For every $P \in X$ there is a compact definable neighbourhood $X_P$
 contained in one of the $U_i$.
 Finitely many of these suffice to cover $X$.
 Let $X_1,\dots,X_m$ be such a cover.
 We start with a definable triangulation on~$X_1$
 compatible with $A_i \cap X_1$ and compatible with all $X_1 \cap X_I$
 for $X_I = \bigcap_{i \in I}X_i$ for all $I\subset\{2,\dots,n\}$,
 and assume that is $C^p$ on all open faces. It exists by Lemma~\ref{lem:triangle_affine} (that is to say by \cite{shiota-book}).
 We apply Lemma~\ref{lem:triangle_extend} to obtain a definable triangulation on $X_1 \cup X_2$
 compatible with $A_i \cap (X_1 \cup X_2)$ for all $i$ and $(X_1 \cup X_2) \cap X_I$ for all $I$
 that is also $C^p$ on all open faces.
 We proceed inductively until we have found
 the desired triangulation of $X$.
\end{proof}

\begin{rem}In the case of an $o$-minimal structure where every definable subset of $\R^N$ admits a partition into smooth cells, e.g., in the semi-algebraic case, the construction gives a triangulation of compact manifolds by simplices which are smooth on all faces.
\end{rem}

\end{appendix}

\bibliographystyle{alpha}
\bibliography{periods}

\end{document}